\documentclass[a4paper, 10pt]{amsart}

\usepackage{latexsym, amssymb, amsfonts, eucal}
\usepackage{amsthm, amsmath}
\usepackage{mathrsfs}
\usepackage{hyperref}
\usepackage{leftidx}
\usepackage{enumerate}
\usepackage{calc}

\input xy
\xyoption{all}
\xyoption{ps}
\SelectTips{cm}{}
\CompileMatrices

\newdir{ >}{{}*!/-10pt/\dir{>}}

\DeclareMathOperator{\Ban}{\mathbf{Ban}} 
\DeclareMathOperator{\Csn}{\mathbf{Csn}} 

\newcommand{\gban}[1]{#1{-}\mathbf{Ban}}

\DeclareMathOperator{\id}{id}

\DeclareMathOperator{\Hom}{Hom}

\DeclareMathOperator{\mono}{\rightarrowtail}
\DeclareMathOperator{\epi}{\twoheadrightarrow}

\DeclareMathOperator{\Ker}{Ker}
\DeclareMathOperator{\Coim}{Coim}
\DeclareMathOperator{\imhelp}{Im} 
\renewcommand{\Im}{\imhelp} 
\DeclareMathOperator{\Coker}{Coker}

\DeclareMathOperator{\Ch}{\mathbf{Ch}}

\DeclareMathOperator{\D}{\mathbf{D}}

\DeclareMathOperator{\real}{real}

\DeclareMathOperator{\prtens}{\widehat{\otimes}}

\DeclareMathOperator{\clspan}{\overline{span}}
\DeclareMathOperator{\rel}{rel}

\DeclareMathOperator{\opp}{op}

\newcommand{\mat}[1]{\left[\begin{smallmatrix} #1 \end{smallmatrix}\right]}

\DeclareMathOperator{\scrA}{\mathscr{A}}
\DeclareMathOperator{\scrB}{\mathscr{B}}
\DeclareMathOperator{\scrC}{\mathscr{C}}

\DeclareMathOperator{\scrE}{\mathscr{E}}

\DeclareMathOperator{\scrH}{\mathscr{H}}

\DeclareMathOperator{\bfL}{\mathbf{L}}

\DeclareMathOperator{\bfR}{\mathbf{R}}

\DeclareMathOperator{\bbZ}{\mathbb{Z}}

\makeatletter
\def\@strippedMR{}
\def\@scanforMR#1#2#3\endscan{%
  \ifx#1M\ifx#2R\def\@strippedMR{#3}%
  \else\def\@strippedMR{#1#2#3}%
  \fi\fi}
\renewcommand\MR[1]{\relax\ifhmode\unskip\spacefactor3000 \space\fi
  \@scanforMR#1\endscan
  MR\MRhref{\@strippedMR}{\@strippedMR}}
\makeatother 

\makeatletter
\newcommand\@dotsep{4.5}
\def\@tocline#1#2#3#4#5#6#7{\relax
  \ifnum #1>\c@tocdepth 
  \else
    \par \addpenalty\@secpenalty\addvspace{#2}%
    \begingroup \hyphenpenalty\@M
    \@ifempty{#4}{%
      \@tempdima\csname r@tocindent\number#1\endcsname\relax
    }{%
      \@tempdima#4\relax
    }%
    \parindent\z@ \leftskip#3\relax \advance\leftskip\@tempdima\relax
    \rightskip\@pnumwidth plus1em \parfillskip-\@pnumwidth
    #5\leavevmode\hskip-\@tempdima #6\relax
    \leaders\hbox{$\m@th
      \mkern \@dotsep mu\hbox{.}\mkern \@dotsep mu$}\hfill
    \hbox to\@pnumwidth{\@tocpagenum{#7}}\par
    \nobreak
    \endgroup
  \fi}
\makeatother 

\newtheoremstyle{mythm}
 {6pt}
 {6pt}
 {\itshape}
 {}
 {\scshape}
 {.}
 {.5em}
 {}%

\newtheoremstyle{mydef}
 {6pt}
 {6pt}
 {}
 {}
 {\scshape}
 {.}
 {.5em}
 {}%
 
\makeatletter
\renewenvironment{proof}[1][\proofname]{\par
  \pushQED{\qed}%
  \normalfont \topsep6\p@\@plus6\p@\relax
  \trivlist
  \item[\hskip\labelsep
        \scshape
    #1\@addpunct{.}]\ignorespaces
}{%
  \popQED\endtrivlist\@endpefalse
}
\makeatother

\theoremstyle{mythm}
\newtheorem{Thm}{Theorem}[section]

\newtheorem{Cor}[Thm]{Corollary}
\newtheorem{Prop}[Thm]{Proposition}

\newtheorem*{Thmnn}{Theorem}

\newtheorem*{Propnn}{Proposition}

\theoremstyle{mydef}

\newtheorem{Exm}[Thm]{Example}

\newtheorem{Not}[Thm]{Notation}
\newtheorem{Rem}[Thm]{Remark}

\newtheorem{Def}[Thm]{Definition}

\newtheorem*{Remnn}{Remark}

\title[On the Duality between $\ell^{1}$-homology and Bounded %
Cohomology]%
{On the Duality between $\ell^{1}$-Homology\\and Bounded Cohomology}
\author{Theo B\"uhler}
\address{Department of Mathematics, ETH Z\"{u}rich, Switzerland}
\email{theo@math.ethz.ch}
\thanks{}
\date{\today}

\begin{document}


\begin{abstract}
  We modify the definition of $\ell^{1}$-homology and argue why our
  definition is more adequate than the classical one. While we cannot
  reconstruct the classical $\ell^{1}$-homology from the new
  definition for various reasons, we can reconstruct its
  Hausdorff\mbox{}ification so that no information concerning
  semi-norms is lost. We obtain an axiomatic characterization
  of our $\ell^{1}$-homology as a universal $\delta$-functor and
  prove that it is pre-dual to our definition of bounded
  cohomology. We thus answer a question raised by L\"oh in her
  thesis. Moreover, we prove Gromov's theorem and the
  Matsumoto-Morita conjecture in our context.
\end{abstract}

\if{0}
  The purpose of this short note is to provide an answer to a question
  raised in L\"oh's thesis~\cite[p.viii]{toolongthesis} concerning the
  duality between $\ell^{1}$-homology and bounded cohomology of
  groups or spaces. As L\"oh
  points out, such a duality cannot hold if the cohomology of a
  complex of Banach spaces is taken in the na\"ive sense, that is to
  say, if it is considered as a (complete) seminormed
  space. However, the theory of abstract truncation provides such a
  duality in a completely formal way: the duality functor on $\Ban$
  yields an exact functor from the heart of the canonical
  right $t$-structure to the heart of the canonical left $t$-structure
  on $\D{(\Ban)}$ and the image of $\ell^{1}$-homology under this
  functor is bounded cohomology.
\fi

\maketitle

\tableofcontents

\section{Introduction}

Gromov introduced $\ell^{1}$-homology and bounded cohomology for
topological spaces in the late
seventies~\cite{MR686042}. The initial purpose of these exotic (co-)homology
theories was to provide topological invariants which control the
minimal volume of a smooth manifold which, by definition, is an
invariant of the differentiable structure. One of Gromov's deeper
theorems asserts that the bounded cohomology of a countable and
connected CW-complex is an invariant of its fundamental group. In
order to make this statement precise, he needed to introduce
$\ell^{1}$-homology and bounded cohomology for discrete groups, which
apparently was developed in unpublished work of Trauber.

Matsumoto-Morita raised the question whether the analog of Gromov's
theorem holds true for $\ell^{1}$-homology
\cite[Remark~2.6]{MR787909}. After some flawed attempts to prove this true,
see \cite{MR2068142} and \cite{MR2142264}, the question was finally answered
affirmatively by L\"oh~\cite{toolongthesis} and the present
author \cite{mythesis} independently.

The variants of $\ell^{1}$-homology and bounded cohomology for groups
were studied by \cite{MR787909} and bounded cohomology was given a
``functorial approach'' by Brooks~\cite{MR624804}, 
Ivanov~\cite{MR806562,MR964260} and
Noskov~\cite{MR1086449,MR1175809}, see
\cite{MR1320279,MR1445197} and \cite{toolongthesis} for further
references. The theory was substantially improved and 
generalized to topological groups by Burger and Monod, 
see~\cite{MR1694584,MR1840942,MR1911660}. While the Burger-Monod
theory proved to be extremely fruitful in the context of rigidity theory,
the algebraic underpinning remained rather undeveloped. In particular, 
it was unknown whether bounded cohomology could be
interpreted as a derived functor. The main purpose of~\cite{mythesis}
is to close this gap and to give an interpretation
of $\ell^{1}$-homology and bounded cohomology in the context of 
modern homological algebra in order to benefit from the power of
its proper language, i.e., category theory.

Let us turn to mathematics proper. Let $\Ban$ be the additive category
of Banach spaces and continuous linear maps. It is well-known that
$\Ban$ is quasi-abelian and that there are enough projectives and
enough injectives. If $G$ is a group, we denote the category of
isometric representations of $G$ on Banach spaces and $G$-equivariant
continuous linear maps by $\gban{G}$. It is easy to prove that
$\gban{G}$ is quasi-abelian and has enough projectives and enough
injectives, hence the formalism of derived categories allows us to
derive functors defined on $\gban{G}$. In order to speak of homology,
the theory of $t$-structures and their hearts is virtually forced upon
us.

For every quasi-abelian category there are \emph{two} canonical
$t$-structures, which we call the \emph{left} and \emph{right}
$t$-structures, see Definition~\ref{def:can-t-structures}.
The left $t$-structure on $\D{(\scrA^{\opp})} \cong 
\left(\D{(\scrA)}\right)^{\opp}$ is dual to the right $t$-structure on
$\D{(\scrA)}$ in the sense of \cite[1.3.2~(iii)]{MR751966}. In
particular the heart $\scrC_{\ell}{(\scrA^{\opp})}$ 
of the left $t$-structure on $\D{(\scrA^{\opp})}$
is equivalent to the opposite category of the heart
$\scrC_{r}{(\scrA)}$ of the right $t$-structure on
$\D{(\scrA)}$. We write
$H_{\ell}^{n}: \D{(\scrA)} \to \scrC_{\ell}{(\scrA)}$ and 
$H_{r}^{n}: \D{(\scrA)} \to \scrC_{r}{(\scrA)}$ for the associated
homological functors.

There is the following explicit description of the left heart
$\scrC_{\ell}{(\scrA)}$ on $\D{(\scrA)}$: objects are represented by 
a monic $(A^{-1} \to A^{0})$ in $\scrA$ while the
morphisms are obtained from the morphisms of pairs by dividing out the
homotopy equivalence relation and inverting quasi-isomorphisms
(bicartesian squares) formally. By the aforementioned duality, the
right heart $\scrC_{r}{(\scrA)}$ has a dual description.

There are exact inclusion functors
\[
\iota_{\ell}: \scrA \to \scrC_{\ell}{(\scrA)} \qquad \text{and} \qquad 
\iota_{r}: \scrA \to \scrC_{r}{(\scrA)}
\]
given on objects by $\iota_{\ell}{(A)} = (0 \to A)$ and 
$\iota_{r}(A) = (A \to 0)$. The functor $\iota_{\ell}$ has a left
adjoint $q_{\ell}$ given on objects by $q_{\ell}(d:A^{-1} \to A^{0}) =
\Coker_{\scrA}{(A)}$. Similarly, $\iota_{r}$ has a right adjoint
$q_{r}$ induced by the kernel functor in $\scrA$.

Let us specialize to the category $\gban{G}$. The trivial module
functor (augmentation) ${}_{\varepsilon}({-}): \Ban \to \gban{G}$ has a
left adjoint given by the co-invariants $({-})_{G}$ and a right 
adjoint given by the invariants $({-})^{G}$. 
Underlying our definition of $\ell^{1}$-homology and
bounded cohomology are the derived functors
\[
\bfL^{-}{({-})_{G}}: \D^{-}{(\gban{G})} \to \D^{-}{(\Ban)}
\]
and
\[
\bfR^{+}{({-})}^{G}: \D^{+}{(\gban{G})} \to \D^{+}{(\Ban)}.
\]
By considering $\gban{G}$ as the full subcategory of complexes
concentrated in degree zero we define for each $M \in \gban{G}$ the
\emph{$\ell^{1}$-homology} of $G$ with coefficients in $M$ as 
\[
\scrH^{\ell^{1}}_{n}{(G,M)} := H^{-n}_{r}{(\bfL^{-}{({-})_{G}}(M))}
\in \scrC_{r}{(\Ban)}
\]
and the \emph{bounded cohomology} of $G$ with coefficients in $M$ as
\[
\scrH_{b}^{n}{(G,M)} := H^{n}_{\ell}{(\bfR^{+}{({-})^{G}}(M))} \in
\scrC_{\ell}{(\Ban)}.
\]
\begin{Thmnn}
  \leavevmode
  \begin{enumerate}[(i)]
    \item
      The $\ell^{1}$-homology functors assemble to a universal 
      homological $\delta$-functor
      \[
      \scrH^{\ell^{1}}_{\ast}{(G, {-})} : \gban{G} \to
      \scrC_{r}{(\Ban)},
      \]
      moreover, $\scrH_{0}^{\ell^{1}}{(G,M)} = (M_G \to 0)$.
    \item
      The bounded cohomology functors assemble to a universal 
      cohomological $\delta$-functor
      \[
      \scrH^{\ast}_{b}{(G, {-})} : \gban{G} \to
      \scrC_{\ell}{(\Ban)},
      \]
      moreover, $\scrH^{0}_{b}{(G,M)} \cong (0 \to M^{G})$.
  \end{enumerate}
\end{Thmnn}
\begin{proof}
  A more precise statement for $\ell^{1}$-homology is given in 
  Theorem~\ref{thm:char-l1-homology} and the (entirely dual) statement
  for bounded cohomology is given in~\cite[p.xiv]{mythesis}.
\end{proof}

\begin{Remnn}
  While it may be perfectly plausible that for duality reasons one
  should choose to use both the left heart and the right heart for defining
  $\ell^{1}$-homology and bounded cohomology, it is natural to wonder
  whether one could interchange ``left'' and ``right'' in the
  definition. In brief, the answer is ``yes, one could, but only at the
  cost of a reasonable duality theory''.
  We will discuss this matter in Section~\ref{sec:remarks}.
\end{Remnn}

The duality functor on $\Ban$ which is exact by Hahn-Banach,
yields an exact duality functor
on $\gban{G}$ and a duality functor
\[
({-})^{\ast}: \scrC_{r}{(\gban{G})} \to \scrC_{\ell}{(\gban{G})}
\]
which is explicitly given on objects by
$(f: A \to B)^{\ast} = (f^{\ast}:B^{\ast} \to A^{\ast})$. We will
prove the following result as Proposition~\ref{prop:duality}.

\begin{Propnn}
  The duality functor $({-})^{\ast}: \scrC_{r}{(\gban{G})}
  \to \scrC_{\ell}{(\gban{G})}$ is well-defined, 
  exact and there is a natural isomorphism of
  functors on $\D{(\gban{G})}$
  \[
  H^{n}_{\ell}{(({-})^{\ast})} \cong \left(H_{r}^{-n}{({-})}\right)^{\ast}.
  \]
\end{Propnn}

One principal motivation for our definition is that one \emph{cannot}
interchange ``left'' and ``right'' in the previous proposition, see 
Remark~\ref{rem:duality}. Theorem~\ref{thm:duality} is:

\begin{Thmnn}
  The duality functor
  $({-})^{\ast}: \scrC_{r}{(\Ban)} \to \scrC_{\ell}{(\Ban)}$ yields a
  natural isomorphism
  \[
  \left( \scrH_{n}^{\ell^{1}}{(G, M)} \right)^{\ast}
  \cong \scrH_{b}^{n}{(G, M^{\ast})}.
  \]
\end{Thmnn}

To end this introductory section, we pass from groups to spaces.
Following Gromov we associate to a
topological space $X$ its $\ell^{1}$-singular chain complex
$C^{\ell^{1}}_{\ast}{(X)}$ and its bounded singular cochain complex
$C^{\ast}_{b}{(X)}$, see~\cite[p.xxi]{mythesis} for the precise 
definition. We define $\ell^{1}$-homology of $X$ as
\[
\scrH_{n}^{\ell^{1}}{(X)} := H_{r}^{-n}{(C^{\ell^{1}}_{\ast}{(X)})}
\in \scrC_{r}{(\Ban)}
\]
and bounded cohomology as 
\[
\scrH^{n}_{b}{(X)} := H_{\ell}^{n}{(C_{b}^{\ast}{(X)})} \in \scrC_{\ell}{(\Ban)}.
\]
If $X$ is a countable and connected CW-complex, let
$G = \pi_{1}{(X)}$ be its fundamental group. We proved
that $C^{\ell^{1}}_{\ast}{(\widetilde{X})}$ considered as complex in
$\gban{G}$ is a projective resolution of the
ground field $k$,
see~\cite[p.xxiii]{mythesis}. Dually, considered as a complex in $\gban{G}$
the bounded cochain complex $C_{b}^{\ast}{(\widetilde{X})}$ is an
injective resolution of the ground field. Our proof of these facts 
relies on one of the main results of Ivanov's proof of Gromov's 
theorem, whence the hypothesis that $X$ be countable. Since 
\[
C^{\ell^{1}}_{\ast}{(X)} \cong (C^{\ell^{1}}_{\ast}{(\widetilde{X})})_{G} \cong
\bfL^{-}{({-})_{G}}(k)
\quad \text{and} \quad 
C_{b}^{\ast}{(X)} \cong (C_{b}^{\ast}{(\widetilde{X})})^{G} \cong 
\bfR^{+}{({-})^{G}}(k)
\]
we obtain the following variant of Gromov's theorem and the
Matsumoto-Morita conjecture:

\begin{Thmnn}
  Let $X$ be a connected and countable CW-complex and let $G =
  \pi_{1}{(X)}$ be its fundamental group. There are canonical isomorphisms:
  \[
  \scrH^{\ell^{1}}_{\ast}{(X)} \cong \scrH^{\ell^{1}}_{\ast}{(G,k)}
  \qquad \text{and} \qquad
  \scrH_{b}^{\ast}{(X)} \cong \scrH_{b}^{\ast}{(G,k)}.
  \]
\end{Thmnn}

\begin{Remnn}
  Notice that we deduced the theorem from the fact that the
  \emph{complexes} computing $\ell^{1}$-homology and bounded
  cohomology are invariants of the fundamental group in the derived
  category $\D{(\Ban)}$.
\end{Remnn}

\begin{Remnn}
  For connected (countable) CW-complexes, 
  L\"oh introduced $\ell^{1}$-homology and bounded cohomology with
  \emph{twisted coefficients}, see \cite[p.27]{toolongthesis}.
  Let $M$ be a Banach
  $G$-module, equip the projective tensor product
  complex $C^{\ell^{1}}_{\ast}{(\widetilde{X})} \prtens M$ 
  with the diagonal $G$-action and apply the co-invariants. In other
  words, she considers
  \[
  C^{\ell^{1}}_{\ast}{(\widetilde{X})} \prtens_{G} M \cong 
  k \prtens_{G}^{\bfL^{-}} M,
  \]
  where the right hand side shows that this complex is an invariant of the
  fundamental group in $\D{(\Ban)}$.
  Similarly, for bounded cohomology, she considers
  the complex
  \[
  \Hom_{G}{(C_{\ast}^{\ell^{1}}(\widetilde{X}), M)} \cong
  \bfR^{+}\Hom_{G}{(k,M)}.
  \]
  Using the facts that $k \prtens_{G} {-} \cong ({-})_G$ and
  $\Hom_{G}{(k,{-})} \cong ({-})^{G}$ as well as the balance of
  the derived tensor product and derived $\Hom$, we immediately conclude that
  these complexes compute
  $\scrH^{\ell^{1}}_{\ast}{(G,M)}$ and $\scrH_{b}^{\ast}{(G,M)}$.
\end{Remnn}

\begin{Remnn}
  The previous remark and our duality theorem constitute a
  rather trivial \emph{universal coefficient theorem} for
  $\ell^{1}$-homology and bounded cohomology with twisted
  coefficients of countable and connected CW-complexes---provided
  that one is willing to accept
  our definition of $\ell^{1}$-homology as the correct one.
\end{Remnn}

\section{Cohomology in Quasi-Abelian Categories}

Let $\scrA$ be an abelian category and consider a complex
\[
A^{\bullet} = (A' \xrightarrow{f} A \xrightarrow{g} A'')
\]
in $\scrA$, that is, $gf = 0$. Since the compositions $\Im{f} \mono A
\to A''$ and $A' \to A \epi \Im{g}$ are both zero we obtain a
commutative diagram
\[
\xymatrix{
  & \Im{f} \ar@{ >->}[rr]^{\varphi} \ar@{ >->}[dr] & & \Ker{g} \ar@{
    >->}[dl] \\
  A' \ar@{->>}[ur] \ar[rr]^{f} & & A \ar@{->>}[dl] \ar@{->>}[dr]
  \ar[rr]^{g} & & A'' \\
  & \Coker{f} \ar@{->>}[rr]^{\psi} & & \Im{g} \ar@{ >->}[ur]
}
\]
and the (co)homology of $A^{\bullet}$ is defined to be any one of the
isomorphic objects
\[
H(A^{\bullet}) \cong \Coker{\varphi} \cong \Ker{\psi} \cong \Im{u},
\]
where $u$ is the morphism $\Ker{g} \to \Coker{f}$, see 
e.g.~\cite[p.178]{MR2182076}.

Recall the notion of a quasi-abelian category in the sense of 
Yoneda~\cite{MR0225854} (see also Prosmans~\cite{MR1749013} and
Schneiders~\cite{MR1779315}): an additive category $\scrA$ is called
\emph{quasi-abelian} if
\begin{enumerate}[(i)]
  \item
    every morphism has a kernel and a cokernel,
  \item
    the class of all kernel-cokernel pairs in
    $\scrA$ is an exact structure in the sense of
    Quillen~\cite{MR0338129}: every kernel is the kernel of its
    cokernel, the class of kernels is closed under composition 
    and push-outs along arbitrary morphisms and, dually, every
    cokernel is the cokernel of its kernel, the class of cokernels is
    closed under composition and pull-backs along arbitrary morphisms.
\end{enumerate}
If $\scrA$ is quasi-abelian but not abelian, the situation is no longer as
straightforward as before. Assume for simplicity that
$\scrA$ has enough projectives and enough injectives. We obtain
the diagram
\[
\xymatrix{
  & \Coim{f} \ar@{ >.>}[rr]^{\varphi} \ar@{ >.>}[dr] & & \Ker{g} \ar@{
    >->}[dl] \\
  A' \ar@{->>}[ur] \ar[rr]^{f} & & A \ar@{->>}[dl] \ar@{.>>}[dr]
  \ar[rr]^{g} & & A'' \\
  & \Coker{f} \ar@{.>>}[rr]^{\psi} & & \Im{g} \ar@{ >->}[ur]
}
\]
in which the dotted arrows are categorical monics or epics (here we
use that there are enough projectives and enough injectives) that may or
may not be kernels or cokernels.

\begin{Rem}[Huber]
  \label{rem:huebi}
  The morphism $u: \Ker{g} \to \Coker{f}$ is \emph{strict} in
  the sense that it factors as $\Ker{g} \epi X \mono \Coker{f}$, so
  that $X \cong \Coim{u} \cong \Im{u}$.
  
  Since $f$ factors over $\Im{f}$ and $gf = 0$, the morphism
  $\Im{f} \mono A$ factors over $\Ker{g} \mono A$.
  The morphism $v:\Im{f} \to \Ker{g}$ is an admissible
  monic by Quillen's ``obscure axiom'',
  see~\cite[A.1, c)$^{\opp}$]{MR1052551}. Let $X = \Coker{v}$ and form
  the following push-out diagram
  \[
  \xymatrix{
    \Ker{g} \ar@{ >->}[d] \ar@{->>}[r] &
    X \ar@{ >->}[d] \\
    A \ar[r] & Y
  }
  \]
  which by \cite[A.1, 1st step]{MR1052551} is bicartesian. It is easy
  to see that $A \to Y$ is the cokernel of $\Im{f} \mono A$ so that 
  $Y \cong \Coker{f}$ (it is a general fact that in an exact category 
  the push-out of an admissible epic along an admissible monic yields 
  an admissible epic). From this diagram one readily reads off
  that
  \[
  \Ker{u} \cong \Im{f} \quad \text{and} \quad \Coker{u} \cong \Coim{g},
  \]
  so $X \cong \Coim{u} \cong \Im{u}$ as claimed.
\end{Rem}

\begin{Rem}
  \label{rem:huebi-2}
  The object $X$ constructed in the previous remark is at the same
  time the cokernel of $\Im{f} \mono \Ker{g}$ and the kernel of
  $\Coker{f} \epi \Coim{g}$. If the quasi-abelian category $\scrA$ is
  such that for each morphism $h$ the morphism
  $\Coim{h} \to \Im{h}$ is categorically monic and epic
  then it follows that 
  $\Coker{\varphi} \cong X \cong \Ker{\psi}$. This is the case if
  $\scrA$ has enough projective and enough injective objects, however, the
  author does not know whether this is true in general.
\end{Rem}

\begin{Exm}
  Let $\scrA = \Ban$ be the category of Banach spaces and consider 
  the complex
  \[
  \ell^{1} \xrightarrow{\mat{i \\ 0}} c_{0} \oplus \ell^{1}
  \xrightarrow{\mat{0 & i}} c_{0}
  \]
  where $i: \ell^{1} \to c_{0}$ is the obvious inclusion. We have
  \[
  \Coim{\mat{i \\ 0}} = \ell^{1}, \quad
  \Ker{\mat{0 & i}} = c_{0}, \quad
  \Coker{\mat{i \\ 0}} = \ell^{1}, \quad
  \Im{\mat{0 & i}} = c_{0},
  \]
  which shows that the dotted morphisms are indeed not kernels or
  cokernels in general.
\end{Exm}

By the theory of $t$-structures, both $\varphi$ and $\psi$ yield
legitimate notions of cohomology: $\varphi$ represents
$H_{\ell}^{0}(A^{\bullet})$ in the left heart $\scrC_{\ell}{(\scrA)}$ and 
$\psi$ represents $H_{r}^{0}{(A^{\bullet})}$ in the right heart
$\scrC_{r}{(\scrA)}$ of the derived category $\D{(\scrA)}$ if
$A^{\bullet}$ is considered as a complex concentrated in degrees
$-1,0,1$. To be more specific, we need two definitions.

\begin{Def}
  \label{def:truncation-functors}
  Let $A^{\bullet} =
  (\cdots \xrightarrow{} A^{-2} \xrightarrow{d^{-2}} A^{-1} 
  \xrightarrow{d^{-1}} A^{0} \xrightarrow{d^{0}} A^{1}
  \xrightarrow{} \cdots)$ 
  be a complex in the quasi-abelian category
  $\scrA$. The \emph{left truncation functors} are defined by
  \[
  \tau_{\ell}^{\leq 0}A^{\bullet} = 
  (\cdots \xrightarrow{} A^{-2} \xrightarrow{d^{-2}} A^{-1} 
  \xrightarrow{} \Ker{d^{0}} \xrightarrow{} 0
  \xrightarrow{} 0 \xrightarrow{} \cdots)
  \]
  and
  \[
  \tau_{\ell}^{\geq 0}A^{\bullet} =
  (\cdots \xrightarrow{} 0 \xrightarrow{} \Coim{d^{-1}} \xrightarrow{}
  A^{0} \xrightarrow{d^{0}} A^{1} 
  \xrightarrow{} \cdots)
  \]
  while the \emph{right truncation functors} are given by
  \[
  \tau_{r}^{\leq 0} A^{\bullet} =
  (\cdots \xrightarrow{} A^{-1} \xrightarrow{d^{-1}} A^{0}
  \xrightarrow{} \Im{d^{0}} \xrightarrow{} 0 \xrightarrow{} \cdots)
  \]
  and
  \[
  \tau_{r}^{\geq 0} A^{\bullet} = (\cdots \xrightarrow{} 0
  \xrightarrow{} \Coker{d^{-1}} \xrightarrow{} A^{1}
  \xrightarrow{d^{1}} A^{2} \xrightarrow{} \cdots).
  \]
  The truncation functors yield endofunctors of the derived category
  $\D{(\scrA)}$. As usual, we put for $n \in \bbZ$
  \[
  \tau_{\ell}^{\leq n} = \Sigma^{-n} \circ \tau_{\ell}^{\leq 0} \circ
  \Sigma^{n},
  \]
  \emph{etc.}
\end{Def}

\begin{Def}
  \label{def:can-t-structures}
  Denote by $\D^{\leq 0}_{\ell}{(\scrA)}$ the essential
  image of $\tau^{\leq 0}_{\ell}$, \emph{etc.} It is not difficult to prove that
  $(\D^{\leq 0}_{\ell}{(\scrA)}, \D^{\geq 0}_{\ell}{(\scrA)})$ is a
  \emph{$t$-structure}, see \cite[1.3.1, 1.3.22]{MR751966}, which we
  call the \emph{left} $t$-structure. By duality
  $(\D^{\leq 0}_{r}{(\scrA)}, \D^{\geq 0}_{r}{(\scrA)})$ is a
  $t$-structure as well and we call it the \emph{right} $t$-structure.
  The corresponding (left and right) \emph{hearts} are
  \[
  \scrC_{\ell}{(\scrA)} = 
  \D^{\leq 0}_{\ell}{(\scrA)} \cap \D^{\geq 0}_{\ell}{(\scrA)} \qquad
  \text{and} \qquad
  \scrC_{r}{(\scrA)} = 
  \D^{\leq 0}_{r}{(\scrA)} \cap \D^{\geq 0}_{r}{(\scrA)},
  \]
  they are admissible abelian subcategories of $\D{(\scrA)}$. The
  associated homological functors are
  $H^{0}_{\ell} = \tau^{\leq 0}_{\ell} \tau^{\geq 0}_{\ell}:
  \D{(\scrA)} \to
  \scrC_{\ell}{(\scrA)}$ and 
  $H^{0}_{r} = \tau_{r}^{\leq 0} \tau_{r}^{\geq 0}$. 
\end{Def}

There is the following explicit description of
$\scrC_{\ell}{(\scrA)}$: objects are represented by 
a monic $(A^{-1} \hookrightarrow A^{0})$ in $\scrA$ while the
morphisms are obtained from the morphisms of pairs by dividing out the
homotopy equivalence relation and inverting quasi-isomorphisms
(bicartesian squares) formally, see \cite[1.3.22]{MR751966}, 
\cite[1.5.7]{MR726427} or
\cite[Construction~2.2.1, p.35]{mythesis}.

\begin{Prop}
  \label{prop:inclusion-functor}
  The \emph{inclusion functor}
  $\iota_{\ell}: \scrA \to \scrC_{\ell}{(\scrA)}$ given on
  objects by $A \mapsto (0 \hookrightarrow A)$ preserves monics, is
  fully faithful, exact and reflects exactness. Its image is closed
  under extensions in $\scrC_{\ell}{(\scrA)}$. It has a left adjoint
  $q_{\ell}$ given on objects by $\Coker{(A^{-1} \hookrightarrow A^{0})}$.
  Every exact and monic-preserving functor $\scrA \to \scrB$ to an
  abelian category factors uniquely over an exact functor
  $\scrC_{\ell}{(\scrA)} \to \scrB$.
\end{Prop}

\begin{proof}
  This is all well-known, see e.g. \cite[Chapter~III.2]{mythesis}.
\end{proof}

\section{$\ell^{1}$-Homology and Bounded Cohomology}

Let $G$ be a group and let $\gban{G}$ be the category of isometric
representations of $G$ on Banach spaces and $G$-equivariant bounded
linear maps. It is a simple consequence of the open mapping theorem
that $\gban{G}$ is quasi-abelian.

\begin{Not}
  Let $\ell^{1}{(G)}$ be the Banach
  group algebra and let $E$ be a Banach space. 
  The \emph{induced Banach $G$-module} is
  \[
  \uparrow\!E = \ell^{1}{(G)} \prtens E \cong \ell^{1}{(G,E)}
  \]
  with the left $G$-action on the factor $\ell^{1}{(G)}$. The
  \emph{coinduced Banach $G$-module} is
  \[
  \Uparrow\!E := 
  \Hom_{\Ban}{(\ell^{1}{(G)}, E)} \cong \ell^{\infty}{(G,E)}
  \]
  with the action coming from the right action of $G$ on
  $\ell^{1}{(G)}$.
\end{Not}

\begin{Not}
  Let $M \in \gban{G}$ be a Banach $G$-module.
  The \emph{module of coinvariants} of
  $M$ is the Banach space
  \[
  M_{G} = M / \clspan\{m - gm\,:\,m \in M, \, g \in G\}
  \]
  and the \emph{module of invariants} is the Banach space
  \[
  M^{G} = \{m \in M\,:\,\text{$gm = m$ for all $g \in G$}\}.
  \]
\end{Not}

At the heart of the homological algebra of $\ell^{1}$-homology and
bounded cohomology is the following simple result which is
proved by direct inspection:

\begin{Thm}[{Fundamental Adjunctions~\cite[p.xviii]{mythesis}}]
  \label{thm:fundamental-adjunctions}
  Let $\downarrow: \gban{G} \to \Ban$ be the forgetful functor and let 
  $\leftidx_{\varepsilon}{({-})}: \Ban \to \gban{G}$ be the trivial
  module functor. There are two adjoint triples of functors
  \[
  \vcenter{
    \xymatrix{
      \gban{G} \ar[d]^-{\downarrow} \\
      \Ban 
      \ar@<-2ex>@/_1.5pc/[u]_-{\Uparrow} 
      \ar@<2ex>@/^1.5pc/[u]^-{\uparrow}
    }
  }
  \qquad \text{and} \qquad
  \vcenter{
    \xymatrix{
      \gban{G} 
      \ar@<-2ex>@/_1.5pc/[d]_-{({-})_{G}} 
      \ar@<2ex>@/^1.5pc/[d]^-{({-})^{G}} \\
      \Ban \ar[u]_-{\leftidx_{\varepsilon}{({-})}}
    }
  }
  \]
  that is to say $\uparrow$ is left adjoint to $\downarrow$ and
  $\downarrow$ is left adjoint to $\Uparrow$, etc.

  The forgetful functor, induction, coinduction are all exact as well
  as the trivial module functor. \qed
\end{Thm}

The most important consequence for the present work is:

\begin{Cor}[{\cite[p.xviii]{mythesis}}]
  There are enough projectives and enough injectives in $\gban{G}$. \qed
\end{Cor}

This allows us to consider the derived functors
\[
\bfL^{-}{({-})_{G}}: \D^{-}{(\gban{G})} \to \D^{-}{(\Ban)}
\]
and
\[
\bfR^{+}{({-})}^{G}: \D^{+}{(\gban{G})} \to \D^{+}{(\Ban)}
\]
which underlie $\ell^{1}$-homology and bounded cohomology.

\begin{Def}
  Let $M \in \gban{G}$. 
  We define \emph{$\ell^{1}$-homology} as
  \[
  \scrH^{\ell^{1}}_{n}{(G, M)} := H^{-n}_{r}{(\bfL^{-}{({-})_{G}}(M))}
  \]
  and \emph{bounded cohomology} as
  \[
  \scrH_{b}^{n}{(G, M)} := H^{n}_{\ell}{(\bfR^{+}{({-})^{G}}(M))}.
  \]
\end{Def}

\begin{Thm}
  \label{thm:char-l1-homology}
  Up to unique isomorphism of $\delta$-functors there is a unique
  family of functors 
  \[
  \scrH^{\ell^{1}}_{n}{(G, {-})} : \gban{G} \to \scrC_{r}{(\Ban)},
  \quad n \in \bbZ,
  \]
  having the following properties:
  \begin{enumerate}[(i)]
    \item
      (Normalization)
      $\scrH_{0}^{\ell^{1}}{(G,M)} = (M_{G} \to 0)$ for all $M \in
      \gban{G}$.

    \item
      (Vanishing)
      $\scrH_{n}^{\ell^{1}}{(G,P)} = 0$ for all projective objects
      $P \in \gban{G}$ and all $n > 0$.

    \item
      (Long exact sequence)
      Associated to each short exact sequence $M' \mono M \epi M''$ in
      $\gban{G}$ there are morphisms
      $\delta_{n+1}:\scrH^{\ell^{1}}_{n+1}{(G,M'')} \to
      \scrH^{\ell^{1}}_{n}{(G,M')}$ depending naturally on the sequence
      and fitting into a long exact sequence
      \[
      \cdots \xrightarrow{\delta_{n+1}}
      \scrH_{n}^{\ell^{1}}{(G,M')} \xrightarrow{}
      \scrH_{n}^{\ell^{1}}{(G,M)} \xrightarrow{}
      \scrH_{n}^{\ell^{1}}{(G,M'')} \xrightarrow{\delta_{n}}
      \scrH_{n-1}^{\ell^{1}}{(G,M')} \xrightarrow{}
      \cdots
      \]
      in $\scrC_{r}{(\Ban)}$.
  \end{enumerate}
\end{Thm}

\begin{proof}[On the Proof]
  This follows from dualizing the proof of the theorem on page~xiv of
  \cite{mythesis}. Notice that $({-})_{G}$ and 
  $\iota_{r}: \Ban \to \scrC_{r}{(\Ban)}, E \mapsto (E \to 0)$ 
  both have a right adjoint. An existence proof is also given in
  Section~\ref{sec:canonical-resolutions}.
\end{proof}

Now consider the duality functor $({-})^{\ast} : \gban{G} \to
\gban{G}$ and recall that it is exact, hence it extends to the derived
category $\D{(\gban{G})}$. It induces a (contravariant) \emph{duality functor}
\[
({-})^{\ast} : \scrC_{r}{(\gban{G})} \to
\scrC_{\ell}{(\gban{G})}
\]
which is explicitly given on objects by $(e: A \to B)^{\ast} = 
(e^{\ast}: B^{\ast} \to A^{\ast})$. 
\if{0}
By elementary duality theory of Banach
spaces~\cite[Chapter~4]{MR1157815}, the functor $({-})^{\ast}$ is exact and
preserves monics; by Proposition~\ref{prop:inclusion-functor} the
inclusion functor $\iota_{\ell}: \Ban \to \scrC_{\ell}{(\Ban)}$ and
hence also the composition 
$\iota_{\ell} \circ ({-})^{\ast}: \Ban^{\opp}
\to \scrC_{\ell}{(\Ban)}$ share these properties. Therefore the composition
$\iota_{\ell} \circ ({-})^{\ast}: \Ban^{\opp} \to \scrC_{\ell}{(\Ban)}$
factors uniquely over an exact functor
$\scrC_{\ell}{(\Ban^{\opp})} \to \scrC_{\ell}{(\Ban)}$. But
$(\scrC_{r}{(\Ban)})^{\opp} \cong \scrC_{\ell}{(\Ban^{\opp})}$ so
that we obtain an exact ``duality functor''
\[
({-})^{\ast}: \scrC_r{(\Ban)}^{\opp} \to \scrC_{\ell}{(\Ban)}
\]
which, by appealing to
\cite[Theorem~2.2.3~(ii), p.37; Remark~2.2.8, p.39]{mythesis}, is
explicitly given on objects by
\[
(e: A^{0} \to A^{1})^{\ast} = (e^{\ast}: (A^{1})^{\ast} \to (A^{0})^{\ast}).
\]
\fi
\begin{Prop}
  \label{prop:duality}
  The duality functor $({-})^{\ast}: \scrC_{r}{(\gban{G})}
  \to \scrC_{\ell}{(\gban{G})}$ is well-defined, 
  exact and there is a natural isomorphism of
  functors on $\D{(\gban{G})}$
  \[
  H^{n}_{\ell}{(({-})^{\ast})} \cong \left(H_{r}^{-n}{({-})}\right)^{\ast}.
  \]
\end{Prop}
\begin{proof}
  First, the duality functor $\scrC_{r}{(\gban{G})} \to
  \scrC_{\ell}{(\gban{G})}$ is well-defined since 
  the duality functor on $\gban{G}$
  \begin{enumerate}[(i)]
    \item
      maps epics (morphisms with dense range) to monics
      (injective morphisms) by \cite[4.12, Corollaries (b), p.99]{MR1157815},
    \item
      preserves the homotopy equivalence relation since it is additive,
    \item
      preserves bicartesian squares because it is exact.
  \end{enumerate}
  Let us prove that the duality functor $\scrC_{r}{(\gban{G})} \to
  \scrC_{\ell}{(\gban{G})}$ is exact. Points (i) and (iii) yield
  that the duality functor
  $({-})^{\ast}: \gban{G}^{\opp} \to \gban{G}$
  is exact and preserves monics. The same holds true for $\iota_{\ell}:
  \gban{G} \to \scrC_{\ell}{(\gban{G})}$, hence also for the
  composition $F = \iota_{\ell} \circ ({-})^{\ast}$. 
  By \cite[2.2.3, p.37]{mythesis} the universal property of the
  inclusion functor $\iota_{\ell}:\gban{G}^{\opp} \to
  \scrC_{\ell}{(\gban{G}^{\opp})}$ yields a unique exact prolongation 
  $\widetilde{F}:\scrC_{\ell}{(\gban{G}^{\opp})} \to
  \scrC_{\ell}{(\gban{G})}$.
  The construction of $\widetilde{F}$ given in 
  \cite[p.40]{mythesis} together with \cite[2.2.8, p.39]{mythesis}
  yield that
  \[
  \widetilde{F}{(f: A \to B)} = (f^{\ast}: B^{\ast} \to  A^{\ast}).
  \]
  so that $\widetilde{F}$ coincides with the above description of the
  duality functor under the equivalence
  $\scrC_{r}{(\gban{G})}^{\opp} \cong
  \scrC_{\ell}{(\gban{G}^{\opp})}$.

  In order to see that there is a natural isomorphism
  $H^{n}_{\ell}{(({-})^{\ast})} \cong
  \left(H^{-n}_{r}{({-})}\right)^{\ast}$,
  it suffices to notice that for a morphism $f$ of $\gban{G}$ there
  are natural isomorphisms
  \[
  (\Coker{f})^{\ast} \cong \Ker{(f^{\ast})}
  \qquad \text{and} \qquad
  (\Im{f})^{\ast} \cong \Coim{(f^{\ast})},
  \]
  which is a straightforward consequence
  of~\cite[4.12, Theorem, p.99]{MR1157815}.
\end{proof}

\begin{Rem}
  \label{rem:duality}
  The dual of a monic in $\gban{G}$ is not in general an epic, the
  range is weak$^{\ast}$-dense 
  by~\cite[4.12, Corollaries, (c), p.99]{MR1157815} but not necessarily
  norm-dense: consider for instance the inclusion 
  $\ell^{1} \hookrightarrow c_{0}$ whose dual is the inclusion
  $\ell^{1} \hookrightarrow \ell^{\infty}$ the range of which is
  clearly not norm-dense. It follows in particular that there is no
  duality functor $\scrC_{\ell}{(\gban{G})} \to \scrC_{r}{(\gban{G})}$
  as constructed above.
  
  In a similar vein, $(\Coim{f})^{\ast}$ does
  not in general coincide with $\Im{(f^{\ast})}$ but rather with its
  weak${}^{\ast}$-closure and $(\Ker{f})^{\ast}$ is isomorphic to
  the codomain modulo the weak${}^{\ast}$-closure of
  $\Im{(f^{\ast})}$, hence it may be distinct from
  $\Coker{(f^{\ast})}$.
\end{Rem}

Recall the main properties of the duality functor on $\gban{G}$:

\begin{Prop}[{\cite[p.65]{mythesis}}]
  The duality functor $({-})^{\ast}:\gban{G} \to \gban{G}$ is exact,
  reflects exactness and sends projective objects to injective
  objects. Moreover, there is a natural isomorphism
  $({-})^{\ast} \circ ({-})_{G} \cong ({-})^{G}
  \circ ({-})^{\ast}$. \qed
\end{Prop}

\begin{Thm}
  \label{thm:duality}
  The duality functor
  $({-})^{\ast}: \scrC_{r}{(\Ban)} \to \scrC_{\ell}{(\Ban)}$ yields a
  natural isomorphism
  \[
  \left( \scrH_{n}^{\ell^{1}}{(G, M)} \right)^{\ast}
  \cong \scrH_{b}^{n}{(G, M^{\ast})}.
  \]
\end{Thm}
\begin{proof}
  To compute $\scrH_{n}^{\ell^{1}}{(G,M)}$ choose a projective
  resolution $P_{\bullet} \epi M$, apply the coinvariants
  $({-})_{G}$ to $P_{\bullet}$ and then the right cohomology functor
  $H_{r}^{-n}$ to the resulting complex. Now the two previous
  propositions give natural isomorphisms
  \[
  (H_{r}^{-n}((P_{\bullet})_G))^{\ast} \cong
  H_{\ell}^{n}{(((P_{\bullet})_{G})^{\ast})} \cong
  H_{\ell}^{n}{(((P_{\bullet})^{\ast})^{G})}
  \]
  and it remains to notice that $M^{\ast} \mono (P_{\bullet})^{\ast}$ is
  an injective resolution of $M^{\ast}$, so that the right hand side 
  computes bounded cohomology in degree~$n$.
\end{proof}

\if{0}
\begin{Prop}
  Consider two composeable morphisms $f,g$ in $\Ban$ such that
  $gf = 0$. In the diagram
  \[
  \xymatrix{
    & \Coim{f} \ar@{ >.>}[rr]^{\varphi} \ar@{ >.>}[dr] & & \Ker{g} \ar@{
      >->}[dl] \\
    A' \ar@{->>}[ur] \ar[rr]^{f} & & A \ar@{->>}[dl] \ar@{.>>}[dr]
    \ar[rr]^{g} & & A'' \\
    & \Coker{f} \ar@{.>>}[rr]^{\psi} & & \Im{g} \ar@{ >->}[ur]
  }
  \]
  we have that
  \[
  \Coker{\varphi} \cong \Coim{u} \cong \Im{u} \cong \Ker{\psi}
  \]
  where $u$ is the composition $\Ker{g} \mono A \epi \Coker{f}$.
\end{Prop}
\begin{proof}
  Since $gf = 0$ we have $\Im{f} \subset \Ker{g}$ and
  thus $\Ker{g} = \Ker{g} + \Im{f}$ in $A$. Clearly,
  \[
  b + \Im{f} \in \Ker{\psi} \quad \Longleftrightarrow \quad
  b \in \Ker{g} + \Im{f} = \Ker{g},
  \]
  so $\Ker{\psi}$ coincides with the set-theoretic image of $u$ which
  is therefore closed and hence
  \[
  \Ker{\psi} \cong \Im{u} \cong \Coim{u}.
  \]
  Moreover, $a \in \Ker{u}$ if and only if $a$ is in the closure of
  the set-theoretic image of $f$, or in other words $a \in
  \Im{\varphi}$. Thus,
  \[
  \Coim{u} \cong \Coker{\varphi}
  \]
  as claimed.
\end{proof}
\fi

\section{Canonical Resolutions}
\label{sec:canonical-resolutions}

Using the canonical resolution associated to the induction
comonad we give a relatively elementary proof
of the existence of the $\ell^1$-homology functors as described in
Theorem~\ref{thm:char-l1-homology}. In the next section we will make
use of this construction in order to relate our theory to the classical
one.

Recall the fundamental adjunction of induction 
${\uparrow} = \ell^{1}{(G)} \prtens {-} : \Ban \to \gban{G}$ to
the forgetful functor $\downarrow: \gban{G} \to \Ban$, see
Theorem~\ref{thm:fundamental-adjunctions}. The latter functor
is obviously exact while the former is exact since $\ell^{1}{(G)}$ is
projective and hence flat as a Banach
space. Every adjoint pair of functors gives rise to a \emph{comonad}
and a \emph{monad,} see~\cite[8.6, 8.7]{MR1269324}, as follows:

Let $L: \scrA \leftrightarrow \scrB : R$ be an adjoint pair and let
$\varepsilon : LR \Rightarrow \id_{\scrB}$ and $\eta: \id_{\scrA}
\Rightarrow RL$ be the adjunction morphisms. Write $\bot = LR$ and
$\top = RL$, as well as  $\delta_{B} = L(\eta_{RB})$ and
$\mu_{A} =  R(\varepsilon_{LA})$, it is then a simple fact that
$(\bot, \varepsilon, \mu)$ is a comonad and $(\top, \eta, \delta)$ is a
monad, see~\cite[8.6.2]{MR1269324}.
The simplicial object associated to the comonad $\bot$
is described in~\cite[8.6.4]{MR1269324}, it gives rise to a simplicial
resolution $\bot_{\ast}B \to B$, where $\bot_{n}B := (\bot)^{n+1}B$.

Suppose $\scrA$ and $\scrB$ are additive. By taking the alternating sum of
the face maps one obtains a complex which we still denote by
$\bot_{\ast}B$, and it yields the \emph{canonical resolution}
$\bot_{\ast} B \to B$. This parlance is justified since it is
well-known and easy to check \cite[8.6.8, 8.6.10]{MR1269324} 
that $R(\bot_{\ast}B) \to R(B)$ as well as
$\bot_{\ast}L(A) \to L(A)$ are chain homotopy equivalences for all
$B \in \scrB$ and all $A \in \scrA$.

We apply this to the \emph{induction comonad} $\bot \, = \,\uparrow
\downarrow$ and obtain in particular for each $M \in \gban{G}$ the
\emph{canonical resolution}
\[
\bot_{\ast} M \to M,
\]
which has the property that for all $M \in \gban{G}$ and all
$E \in \Ban$ the complexes
\[
\downarrow ( \cdots \to \bot_{1} M \to \bot_{0} M \to M ) \qquad
\text{and} \qquad
\cdots \to \bot_{1} {\uparrow}E \to \bot_{0} {\uparrow} E \to
{\uparrow} E
\]
are split exact in $\Ban$ and $\gban{G}$, respectively.

Since $\bot$ is exact, we obtain for each short exact sequence 
$M' \mono M \epi M''$ a short exact sequence of complexes
\[
\bot_{\ast} M' \mono \bot_{\ast} M \epi \bot_{\ast} M''
\]
in $\Ch^{\leq 0}{(\gban{G})}$. Writing temporarily
$\bot_{-1} = \id_{\gban{G}}$ we have for all $n \geq 0$
\[
(\bot_{n} M)_{G} \cong 
\ell^{1}{(G)} \prtens_{G} {}_\varepsilon({\downarrow}\bot_{n-1} M) \cong
{\downarrow}\bot_{n-1} M,
\]
so we get a short exact sequence of complexes in $\Ch^{\leq 0}{(\Ban)}$
\[
(\bot_{\ast} M')_{G} \mono (\bot_{\ast} M)_{G} \epi
(\bot_{\ast} M'')_{G}.
\]

Since the inclusion functor $\iota_{r}: \Ban \to \scrC_{r}{(\Ban)}$ is
exact, the snake lemma provides us with a long exact sequence
\[
\cdots \to 
H^{n}_{r} (\iota_{r}(\bot_{\ast} M')_{G}) \to
H^{n}_{r} (\iota_{r}(\bot_{\ast} M)_{G}) \to
H^{n}_{r} (\iota_{r}(\bot_{\ast} M'')_{G}) \to 
H^{n+1}_{r}(\cdots) \to 
\cdots
\]
which is obviously natural in the short exact sequence
$M' \mono M \epi M''$ so that we have constructed a $\delta$-functor.

Because the complexes involved are concentrated in non-positive degrees
and because $\iota_{r}$ and $({-})_{G}$ are left adjoints and hence 
commute with taking cokernels, we have that
\begin{align*}
  H_{r}^{0}{(\iota_{r}((\bot_{\ast} M)_{G}))} & = 
  \Coker{(\iota_{r} ((\bot_{1} M \to \bot_{0} M)_G))} \\
  &
  \cong \iota_{r} \circ  ({-})_{G} \circ \Coker{(\bot_{1} M \to \bot_{0} M)} \\
  &\cong (M_{G} \to 0).
\end{align*}
For each Banach space $E$ the sequence
\[
\cdots \to \bot_{1} {\uparrow}E \to \bot_{0} {\uparrow} E \to
{\uparrow} E
\]
is split exact, so the map
\[
(\bot_{\ast} {\uparrow}E)_{G} \to ({\uparrow}E)_G \cong E
\]
is a quasi-isomorphism and hence the cohomology of the complex
$\iota_{r}(\bot_{\ast} {\uparrow}E)_{G}$ vanishes outside degree
zero. Finally, the morphism ${\downarrow}\varepsilon_{M} :
{\downarrow}\bot M \to {\downarrow}M$ is a split epic for each
$M \in \gban{G}$, hence $\bot M \to M$ is an admissible epic and it follows
that every projective $P \in \gban{G}$ is a direct summand of
$\bot P = {\uparrow} {\downarrow} P$. Consequently, 
our $\delta$-functor vanishes on
projectives outside degree zero and we conclude from
Theorem~\ref{thm:char-l1-homology} that:

\begin{Thm}
  There is a canonical isomorphism 
  $\scrH^{\ell^{1}}_{\ast}{(G,{-})} \cong 
  H^{-\ast}_{r}{(\iota_{r}(\bot_{\ast}({-}))_G)}$.
\end{Thm}

\begin{Rem}
  The complex $\bot_{\ast}M$ is of course nothing but the
  bar resolution as given e.g. in \cite[(2.13),
  p.20]{toolongthesis}. Call a Banach $G$-module \emph{induced} if it
  is of the form ${\uparrow}E$ for some $E \in \Ban$. 
  By \cite[8.6.7, Exercise 8.6.3]{MR1269324} the direct summands of
  induced modules are precisely the $\bot$-projective objects, or,
  equivalently, the projective objects with respect to the
  exact structure $\scrE^{G}_{\rel}$ on $\gban{G}$ consisting of 
  short sequences $\sigma$ such that ${\downarrow}\sigma$ is
  split exact. This notion is closely related to 
  \emph{relative projectivity} as defined in
  \cite[(A.1), p.104]{toolongthesis} but it is somewhat less
  restrictive.
\end{Rem}

In particular we have shown:

\begin{Cor}
  Every $\bot$-projective object is
  $\scrH^{\ell^{1}}_{\ast}{(G,{-})}$-acyclic. \qed
\end{Cor}

\begin{Rem}
  The acyclicity of $\bot$-projective objects 
  implies by dimension-shifting that one may compute
  $\ell^{1}$-cohomology with coefficients in $M$ using any resolution
  $P_{\bullet} \epi M$ with $\bot$-projective
  components. Requiring that $\epi$ is more than just a
  quasi-isomorphism (e.g., a \emph{strong resolution})
  is only necessary if one is concerned with ensuring that the
  resolution can be used to compute the canonical semi-norms.
\end{Rem}

\begin{Rem}
  The construction given here shows in particular that
  $\ell^{1}$-homology is the
  \emph{derived functor of the induction comonad} with
  \emph{coefficient functor $\iota_{r}$} in the sense of Barr and
  Beck, see e.g.~\cite[8.7.1]{MR1269324}.
\end{Rem}

\begin{Rem}
  Putting $\top = {\Uparrow}{\downarrow}$ we obtain the
  \emph{coinduction monad} which we will not discuss further because
  the arguments given in this section are
  straightforward to dualize.
\end{Rem}

\section{Remarks on our Definition of $\ell^{1}$-Homology}
\label{sec:remarks}

Our first and main motivation for our definition of $\ell^{1}$-homology is
purely utilitarian in nature: we want to have a smooth duality between
$\ell^{1}$-homology and bounded cohomology in order to save a lot of
work.

Second, we want to show that no information concerning semi-norms is
lost:
For this we need to describe the classical
$\ell^{1}$-homology as defined e.g. in \cite{toolongthesis}. 
An object of $\scrC_{\ell}{(\Ban)}$ can be
considered as a morphism of the category $\Csn$ of
\emph{complete seminormed spaces} and continuous linear maps. Taking
the cokernel in $\Csn$ gives a realization functor
$\real: \scrC_{\ell}{(\Ban)} \to \Csn$
which is exact in the sense that it transforms exact sequences to
sequences in $\Csn$ whose underlying sequence of vector spaces is
exact, see~\cite[p.xv, Lemma]{mythesis}. It is thus easy to see that
$\ell^{1}$-homology as defined e.g. in \cite{toolongthesis} coincides with
\[
H^{\ell^{1}}_{n}{(G,M)} = 
\real{H^{-n}_{\ell}(\bot_{\ast}M)}
\]
Notice that we use the \emph{left} homology functor $H_{\ell}^{\ast}$
instead of the right one. We have
\[
H_{b}^{n}{(G,M)}
\cong \real H_{\ell}^{n}{(\top^{\ast}M)}
\cong \real \scrH_{b}^{n}{(G,M)}.
\]
The complications involved in the development of 
a reasonable duality between the two classical
theories is discussed at length in \cite[Chapter~3]{toolongthesis}.

Recall that the inclusion functor 
$\iota_{\ell}: \Ban \to \scrC_{\ell}{(\Ban)}$ has a left adjoint
$q_{\ell}$ defined on objects by taking the cokernel in
$\Ban$, see Proposition~\ref{prop:inclusion-functor}. 
Dually, the inclusion functor $\iota_{r}$ has a right adjoint
given by taking the kernel in $\Ban$. Remark~\ref{rem:huebi-2} implies that
there is a natural isomorphism
$q_{\ell} H^{n}_{\ell} \cong q_{r} H^{n}_{r}$ on the derived category 
$\D{(\Ban)}$. From all this we deduce easily:

\begin{Thm}
  The functor $q_{r} \circ \scrH^{\ell^{1}}_{\ast}{(G,{-})}$ coincides with
  Hausdorff\mbox{}ification of classical $\ell^{1}$-homology
  $H_{\ast}^{\ell^{1}}{(G,{-})}$. Similarly, 
  $q_{\ell}\circ \scrH_{b}^{\ast}{(G,{-})}$ coincides with the
  Haus\-dorff\mbox{}\-ification of classical bounded cohomology
  $H^{\ast}_{b}{(G,{-})}$. \qed
\end{Thm}

\begin{Rem}
  The main interest of the theorem is of course that it shows that as
  far as semi-norms are concerned one may as well work with our
  version of the $\ell^{1}$-homology 
  functors since Hausdorff\mbox{}ification only consists of quotienting
  out the space of vectors of semi-norm zero in
  $H^{\ell^{1}}_{n}{(G,M)}$.
\end{Rem}

\begin{Rem}
  It is important to notice that Hausdorff\mbox{}ification as well as
  $q_{r}$ both fail to be ``exact'', so that the
  Hausdorff\mbox{}ified long exact sequence of $\ell^{1}$-homology and
  bounded cohomology is no longer exact in general.
\end{Rem}

\begin{Rem}
  The dual of the kernel of $f: M \to N$ in $\Ban$
  is \emph{not} the cokernel of the dual map $f^{\ast}:N^{\ast} \to M^{\ast}$
  in general but the quotient of $M^{\ast}$ by the
  weak${}^{\ast}$-closure of the range of $f^{\ast}$. So there is only
  a natural quotient map
  \[
  q_{\ell} \circ \scrH_{b}^{\ast}(G,M^{\ast})
  \epi (q_{r} \circ \scrH_{\ast}^{\ell^{1}}{(G,M)})^{\ast}
  \]
  as is well-known in the classical context, see
  e.g. \cite[p.540]{MR787909}. This map is of course not an
  isomorphism in general.
\end{Rem}


\bibliographystyle{amsalpha}
\bibliography{duality}

\def\cprime{$'$}
\providecommand{\bysame}{\leavevmode\hbox to3em{\hrulefill}\thinspace}
\providecommand{\MR}{\relax\ifhmode\unskip\space\fi MR }
\providecommand{\MRhref}[2]{%
  \href{http://www.ams.org/mathscinet-getitem?mr=#1}{#2}
}
\providecommand{\href}[2]{#2}
\begin{thebibliography}{BBD82}

\bibitem[BBD82]{MR751966}
A.~A. Be{\u\i}linson, J.~Bernstein, and P.~Deligne, \emph{Faisceaux pervers},
  Analysis and topology on singular spaces, I (Luminy, 1981), Ast\'erisque,
  vol. 100, Soc. Math. France, Paris, 1982, pp.~5--171. \MR{MR751966
  (86g:32015)}

\bibitem[BM99]{MR1694584}
M.~Burger and N.~Monod, \emph{Bounded cohomology of lattices in higher rank
  {L}ie groups}, J. Eur. Math. Soc. (JEMS) \textbf{1} (1999), no.~2, 199--235.
  \MR{MR1694584 (2000g:57058a)}

\bibitem[BM02]{MR1911660}
\bysame, \emph{Continuous bounded cohomology and applications to rigidity
  theory}, Geom. Funct. Anal. \textbf{12} (2002), no.~2, 219--280.
  \MR{MR1911660 (2003d:53065a)}

\bibitem[Bou04]{MR2142264}
Abdesselam Bouarich, \emph{Th\'eor\`emes de {Z}ilber-{E}ilemberg \emph{[sic!]}
  et de {B}rown en homologie {$\ell_1$}}, Proyecciones \textbf{23} (2004),
  no.~2, 151--186. \MR{MR2142264 (2006h:55008)}

\bibitem[Bro81]{MR624804}
Robert Brooks, \emph{Some remarks on bounded cohomology}, Riemann surfaces and
  related topics: Proceedings of the 1978 Stony Brook Conference (State Univ.
  New York, Stony Brook, N.Y., 1978) (Princeton, N.J.), Ann. of Math. Stud.,
  vol.~97, Princeton Univ. Press, 1981, pp.~53--63. \MR{MR624804 (83a:57038)}

\bibitem[B{\"u}h08]{mythesis}
Theo B{\"u}hler, \emph{On the algebraic foundation of bounded cohomology},
  Ph.D. thesis, ETH Z\"urich, 2008.

\bibitem[Gri95]{MR1320279}
R.~I. Grigorchuk, \emph{Some results on bounded cohomology}, Combinatorial and
  geometric group theory (Edinburgh, 1993), London Math. Soc. Lecture Note
  Ser., vol. 204, Cambridge Univ. Press, Cambridge, 1995, pp.~111--163.
  \MR{MR1320279 (96j:20073)}

\bibitem[Gri96]{MR1445197}
\bysame, \emph{Bounded cohomology of group constructions}, Mat. Zametki
  \textbf{59} (1996), no.~4, 546--550. \MR{MR1445197 (98f:20037)}

\bibitem[Gro82]{MR686042}
Michael Gromov, \emph{Volume and bounded cohomology}, Inst. Hautes \'Etudes
  Sci. Publ. Math. (1982), no.~56, 5--99 (1983). \MR{MR686042 (84h:53053)}

\bibitem[Iva85]{MR806562}
N.~V. Ivanov, \emph{Foundations of the theory of bounded cohomology}, Zap.
  Nauchn. Sem. Leningrad. Otdel. Mat. Inst. Steklov. (LOMI) \textbf{143}
  (1985), 69--109, 177--178. \MR{MR806562 (87b:53070)}

\bibitem[Iva88]{MR964260}
\bysame, \emph{The second bounded cohomology group}, Zap. Nauchn. Sem.
  Leningrad. Otdel. Mat. Inst. Steklov. (LOMI) \textbf{167} (1988), no.~Issled.
  Topol. 6, 117--120, 191. \MR{MR964260 (90a:55015)}

\bibitem[Kel90]{MR1052551}
Bernhard Keller, \emph{Chain complexes and stable categories}, Manuscripta
  Math. \textbf{67} (1990), no.~4, 379--417. \MR{MR1052551 (91h:18006)}

\bibitem[KS06]{MR2182076}
Masaki Kashiwara and Pierre Schapira, \emph{Categories and sheaves},
  Grundlehren der Mathematischen Wissenschaften [Fundamental Principles of
  Mathematical Sciences], vol. 332, Springer-Verlag, Berlin, 2006.
  \MR{MR2182076 (2006k:18001)}

\bibitem[Lau83]{MR726427}
G.~Laumon, \emph{Sur la cat\'egorie d\'eriv\'ee des {$\mathcal{D}$}-modules
  filtr\'es}, Algebraic geometry (Tokyo/Kyoto, 1982), Lecture Notes in Math.,
  vol. 1016, Springer, Berlin, 1983, pp.~151--237. \MR{MR726427 (85d:32022)}

\bibitem[L{\"o}h07]{toolongthesis}
Clara L{\"o}h, \emph{$\ell^{1}$-{H}omology and {S}implicial {V}olume}, Ph.D.
  thesis, Westf\"alische Wilhelms-Universit\"at M\"unster, 2007.

\bibitem[MM85]{MR787909}
Shigenori Matsumoto and Shigeyuki Morita, \emph{Bounded cohomology of certain
  groups of homeomorphisms}, Proc. Amer. Math. Soc. \textbf{94} (1985), no.~3,
  539--544. \MR{MR787909 (87e:55006)}

\bibitem[Mon01]{MR1840942}
Nicolas Monod, \emph{Continuous bounded cohomology of locally compact groups},
  Lecture Notes in Mathematics, vol. 1758, Springer-Verlag, Berlin, 2001.
  \MR{MR1840942 (2002h:46121)}

\bibitem[Nos90]{MR1086449}
G.~A. Noskov, \emph{Bounded cohomology of discrete groups with coefficients},
  Algebra i Analiz \textbf{2} (1990), no.~5, 146--164. \MR{MR1086449
  (92b:57005)}

\bibitem[Nos92]{MR1175809}
\bysame, \emph{The {H}ochschild-{S}erre spectral sequence for bounded
  cohomology}, Proceedings of the International Conference on Algebra, Part 1
  (Novosibirsk, 1989) (Providence, RI), Contemp. Math., vol. 131, Amer. Math.
  Soc., 1992, pp.~613--629. \MR{MR1175809 (93g:20096)}

\bibitem[Par04]{MR2068142}
HeeSook Park, \emph{Foundations of the theory of {$l\sb 1$} homology}, J.
  Korean Math. Soc. \textbf{41} (2004), no.~4, 591--615. \MR{MR2068142
  (2005c:55011)}

\bibitem[Pro00]{MR1749013}
Fabienne Prosmans, \emph{Derived categories for functional analysis}, Publ.
  Res. Inst. Math. Sci. \textbf{36} (2000), no.~1, 19--83. \MR{MR1749013
  (2001g:46156)}

\bibitem[Qui73]{MR0338129}
Daniel Quillen, \emph{Higher algebraic {$K$}-theory. {I}}, Algebraic
  $K$-theory, I: Higher $K$-theories (Proc. Conf., Battelle Memorial Inst.,
  Seattle, Wash., 1972), Springer, Berlin, 1973, pp.~85--147. Lecture Notes in
  Math., Vol. 341. \MR{MR0338129 (49 \#2895)}

\bibitem[Rud91]{MR1157815}
Walter Rudin, \emph{Functional analysis}, International Series in Pure and
  Applied Mathematics, McGraw-Hill Inc., New York, 1991. \MR{MR1157815
  (92k:46001)}

\bibitem[Sch99]{MR1779315}
Jean-Pierre Schneiders, \emph{Quasi-abelian categories and sheaves}, M\'em.
  Soc. Math. Fr. (N.S.) (1999), no.~76, vi+134. \MR{MR1779315 (2001i:18023)}

\bibitem[Wei94]{MR1269324}
Charles~A. Weibel, \emph{An introduction to homological algebra}, Cambridge
  Studies in Advanced Mathematics, vol.~38, Cambridge University Press,
  Cambridge, 1994. \MR{MR1269324 (95f:18001)}

\bibitem[Yon60]{MR0225854}
Nobuo Yoneda, \emph{On {E}xt and exact sequences}, J. Fac. Sci. Univ. Tokyo
  Sect. I \textbf{8} (1960), 507--576 (1960). \MR{MR0225854 (37 \#1445)}

\end{thebibliography}

\end{document}